\documentclass[12pt,final]{amsart}
\usepackage[utf8]{inputenc}
\usepackage[T1]{fontenc}
\usepackage[letterpaper,centering]{geometry}
\usepackage{lmodern}
\usepackage{mathrsfs}
\usepackage{amsmath,amssymb,amsfonts}
\usepackage[unicode,pdfborder={0 0 0},final]{hyperref}
\usepackage{enumerate}
\usepackage{multicol}
\usepackage{color}

\usepackage[all]{xy}
{\setbox0\hbox{$ $}}\fontdimen16\textfont2=\fontdimen17\textfont2
\entrymodifiers={+!!<0pt,\the\fontdimen22\textfont2>}
\SelectTips{cm}{12}

\theoremstyle{plain}
\newtheorem{thm}{Theorem}[section]

\newtheorem{lem}[thm]{Lemma}
\newtheorem{question}[thm]{Question}
\newtheorem{cor}[thm]{Corollary}
\newtheorem{prop}[thm]{Proposition}
\theoremstyle{definition}
\newtheorem{defn}[thm]{Definition}
\newtheorem{rmk}[thm]{Remark}
\newtheorem{rmks}[thm]{Remarks}

\theoremstyle{remark}
\newtheorem{Step}{Step}
\numberwithin{equation}{section}

\definecolor{vertow}{rgb}{0.18,0.61,0.31}

\newcommand{\isoto}{\myxrightarrow{\,\sim\,}}
\makeatletter
\def\myrightarrow{{\setbox\z@\hbox{$\rightarrow$}\dimen0\ht\z@\multiply\dimen0 6\divide\dimen0 10\ht\z@\dimen0\box\z@}}
\def\myrightarrowfill@{\arrowfill@\relbar\relbar\myrightarrow}
\newcommand{\myxrightarrow}[2][]{\ext@arrow 0359\myrightarrowfill@{#1}{#2}}
\def\myleftarrow{{\setbox\z@\hbox{$\leftarrow$}\dimen0\ht\z@\multiply\dimen0 6\divide\dimen0 10\ht\z@\dimen0\box\z@}}
\def\myleftarrowfill@{\arrowfill@\myleftarrow\relbar\relbar}
\newcommand{\myxleftarrow}[2][]{\ext@arrow 3095\myleftarrowfill@{#1}{#2}}
\makeatother
\newcommand{\kF}{{\mathfrak F}}

\newcommand{\kU}{{\mathfrak U}}

\newcommand{\kX}{{\mathfrak X}}
\newcommand{\kY}{{\mathfrak Y}}
\newcommand{\kZ}{{\mathfrak Z}}
\newcommand{\sC}{{\mathscr C}}

\newcommand{\sF}{{\mathscr F}}

\newcommand{\sO}{{\mathscr O}}

\newcommand{\ok}{{\overline{k}}}
\newcommand{\oM}{{\mkern2mu\overline{\mkern-2muM}}}
\newcommand{\wM}{{\widetilde{M}}}
\newcommand{\Dbar}{{\overline{\D\mkern-1.2mu}\mkern1.2mu}}
\newcommand{\wpsi}{{\widetilde{\psi}}}
\newcommand{\wxi}{{\widetilde{\xi}}}

\newcommand{\C}{{\mathbf C}}
\newcommand{\D}{{\mathbf D}}

\newcommand{\N}{{\mathbf N}}
\renewcommand{\P}{{\mathbf P}}

\newcommand{\Id}{\mathrm{Id}}

\newcommand{\ev}{\mathrm{ev}}

\newcommand{\Spec}{\mathrm{Spec}}
\newcommand{\Hilb}{\mathrm{Hilb}}
\newcommand{\Res}{\mathrm{Res}}

\renewcommand{\emptyset}{\varnothing}

\newcommand{\ci}{\sC^\infty}

\date{March 12th, 2025}
\title{Tight approximation for rationally simply connected varieties}

 \author{Olivier Benoist}
\address{D\'epartement de math\'ematiques et applications, \'Ecole normale sup\'erieure, 45~rue d'Ulm, 75230 Paris Cedex 05, France}
 \email{olivier.benoist@ens.fr}

 \author{Olivier Wittenberg}
\address{Institut Galil\'ee, Universit\'e Sorbonne Paris Nord, 99~avenue Jean-Baptiste Cl\'ement, 93430 Villetaneuse, France}
\email{wittenberg@math.univ-paris13.fr}

\begin{document}
\begin{abstract}
We prove that holomorphic maps from an open subset of a complex smooth projective
curve to a complex smooth projective rationally simply connected variety
can be approximated by algebraic maps for the compact-open topology.
This theorem can be applied in particular when the target is a smooth
hypersurface of degree $d$ in~$\P^n_{\C}$ with~${n\geq d^2-1}$.  We deduce
it from a more general result: the tight approximation property holds for
rationally simply connected varieties over function fields of complex
curves.
\end{abstract}

\dedicatory{To Yuri Prokhorov, on the occasion of his 60th birthday}
\maketitle

\section{Introduction}

Algebraic approximation problems for holomorphic maps are central in complex algebraic geometry.
We will focus on the following precise question.

\begin{question}
\label{quest}
Let $B$ and $X$ be smooth projective algebraic varieties over $\C$. Suppose that $B$ is a connected
curve.  Let $\Omega\subset B(\C)$ be an open subset. Can all holomorphic maps~$\Omega\to X(\C)$ be approximated by algebraic maps (that is, by maps induced by morphisms of algebraic varieties $B\to X$)
for the compact-open topology?
\end{question}

The archetype of such a result is Runge's theorem and its generalizations by Behnke--Stein \cite{BS} and Royden \cite{Royden},
which yield a positive answer to Question~\ref{quest} for~${X=\P^1_{\C}}$ (see \cite[Theorem 10]{Royden}).
Further positive answers were obtained when $X$ is a Grassmannian (see \cite[Theorem 1]{kucharzRunge})
and when $X$ is a rational surface \cite[Corollary 1.2]{BKC}. More generally,  a positive
answer was given to Question~\ref{quest} when~$X$ is birational to an iterated fibration into homogeneous spaces under connected linear algebraic groups (combine \cite[Theorem~3.1, Theorem~5.1 and Theorem~7.4]{bwtight}). This
 class of varieties includes smooth cubic hypersurfaces of dimension $\geq 2$ and smooth complete intersections of two quadrics of dimension~$\geq 2$ (see \cite[Examples~5.4 and~5.6]{bwtight}).
In addition, Forstneri\v{c} and Lárusson have answered Question~\ref{quest}
 positively when~$X$ is birational to an algebraically elliptic smooth projective variety (see \cite[Corollary 1.8]{FL}); in view of \cite[Corollary~1.4]{kalimanzaidenberg}, this recovers the case of smooth cubic hypersurfaces of dimension $\geq 2$. 

The answer to Question~\ref{quest} can also be in the negative, for instance when $B=\P^1_\C$ and~$X$ is an elliptic curve; more generally, a positive answer implies that $b_1(X)=0$ (see \cite[Corollary~1.3]{BKC}).
We do not know the answer for any~$K3$ surface~$X$.

It was
suggested in \cite[Question 3.6]{bwtight} that Question \ref{quest} should have a positive answer whenever $X$ is rationally connected, i.e.\ whenever any two general points of~$X(\C)$ can be joined by a rational curve.  In this article, we provide an answer for \textit{rationally simply connected varieties}. 

\begin{thm}
\label{th1}
Question \ref{quest} has a positive answer whenever $X$ is rationally simply connected.
\end{thm}

The class of rationally simply connected varieties is more restricted than
the class of rationally connected varieties: one not only requires that~$X$
be rationally connected, but also that some spaces of rational curves
on~$X$ themselves be rationally connected.  These varieties were first
introduced by Starr \cite{Starr}; their theory was developed by de Jong and
Starr \cite{djs} and in the far-reaching work of de Jong, He and Starr \cite{dJHS}.  Various definitions have appeared in the literature. We refer
to~\S\ref{defRSC} for the one that we use in this article and that is
close to Hassett's \cite[Theorem~4.7]{Hassett}.

\bigskip
The main concrete applications of Theorem \ref{th1} come from the fact that smooth complete intersections of sufficiently low degree in weighted projective spaces are known to be rationally simply connected, by the works of de Jong and Starr \cite{djs} and Minoccheri \cite{minoccheri}.  In \S\ref{exRSC}, we recall these results and verify their compatibility with the definition of rationally simply connected variety that we use. In the case of hypersurfaces in projective space, we deduce the following corollary.

\begin{cor}
\label{corhypersurface}
Question \ref{quest} has a positive answer for smooth hypersurfaces of degree~$d$ in $\P^n_{\C}$ whenever $n\geq d^2-1$.
\end{cor}

 This was previously only known when $n\geq \phi(d)$ for some function $\phi:\N\to\N$ with superexponential growth
(see \cite[Example 5.5]{bwtight}).

\bigskip
It has been understood since \cite{BKC} that it is better to consider a strengthening of Question \ref{quest} allowing for jet conditions at finitely many points of $\Omega$. (Such jet conditions are taken into account in \cite{bwtight} and \cite{FL}.)
In addition, in \cite{bwtight},
we proposed studying Question~\ref{quest} for
one-parameter families of algebraic varieties rather than for a fixed
algebraic variety. 
Our setup is the following. 

 Let $B$ be a connected smooth projective curve over $\C$. Let $X$ be a smooth projective variety over $\C(B)$ and let  $f:\kX\to B$ be a projective regular model of~$X$. 
We say that~$X$ satisfies the \textit{tight approximation property} if any holomorphic section~$u$ of~$f$ over an open subset $\Omega\subset B(\C)$ can be approximated,  for the compact-open topology,  by algebraic sections that have the same jets as $u$ at finitely many prescribed points of~$\Omega$, at any prescribed finite order (see \cite[Definition~2.1]{bwtight} or Definition~\ref{deftight} below). This property only depends on $X$, not on~$f$ (see \cite[Theorem~3.1]{bwtight}). 
When~$X$ comes, by scalar extension, from a variety~$X_0$ defined over~$\C$, the fibration~$f:\kX\to B$ can be chosen to be the second projection  $X_0 \times B\to B$, and the tight approximation property reduces to a strengthening of Question~\ref{quest} for~$X_0$
that takes into account jet conditions at finitely many points of~$\Omega$. 

Here is the main theorem of this article. In view of the above discussion, it is a generalization of Theorem \ref{th1}.

\begin{thm}
\label{th2}
Let $B$ be a connected smooth projective curve over $\C$.  Any smooth, projective, rationally simply connected variety over $\C(B)$ satisfies the tight approximation property.
\end{thm}

Let us record the following consequence, of which Corollary \ref{corhypersurface} is a special case.

\begin{cor}
\label{corhyptight}
Let $B$ be a connected smooth projective curve over~$\C$.  
Let $d,n\geq 1$ be integers such that $n\geq d^2-1$. 
Any smooth hypersurface of degree~$d$ in $\P^n_{\C(B)}$ satisfies the tight approximation property. 
\end{cor}

We refer to Remarks \ref{remOka1} for a discussion of the relation between the tight approximation property and the Oka-1 and algebraically Oka-1 properties introduced by Alarc\'on--Forstneri\v{c} \cite{AF} and Forstneri\v{c}--L\'arusson \cite{FL}.   For now,  we note that as a consequence of Theorem~\ref{th2},
the algebraically Oka\nobreakdash-1 property, and therefore also the Oka-1 property,
are satisfied by smooth, projective, rationally simply connected varieties over~$\C$ (e.g.\ smooth hypersurfaces of degree~$d$ in~$\P^n_\C$ with $n \geq d^2-1$).

As the tight approximation property incorporates a jet condition, it
implies, as well, the
classical weak approximation property of smooth projective varieties $X$ over~$\C(B)$ (for which see \cite[Section 1]{Hassett}). 
 The weak approximation property is also conjectured to hold whenever~$X$ is rationally connected \cite[Conjecture 2]{HT}; this is known if one imposes jet conditions only at points of~$B$ where~$X$ has good reduction (Hassett--Tschinkel \cite[Theorem~3]{HT}). 

Under the stronger hypothesis that $X$ is rationally simply connected,  Hassett and Starr proved in \cite[Theorem 4.7]{Hassett} that the weak approximation property holds in full generality, without excluding the points of~$B$
where~$X$ has bad reduction.
Our proof of Theorem \ref{th2} is inspired by their argument,  based
on the deformation theory of rational curves and in particular on the Graber--Harris--Starr theorem~\cite{ghs};
in addition, we rely
on a Nash approximation theorem due to Demailly, Lempert and Shiffman \cite[Theorem~1.1]{DLS}.

Here is a more precise outline of the proof of Theorem \ref{th2}.  Choose a projective regular model $f:\kX\to B$ of $X$. Let $u:\Omega\to \kX(\C)$ be the holomorphic section of~$f$ over $\Omega\subset B(\C)$ to be approximated by algebraic maps.  The Nash approximation result \cite[Theorem~1.1]{DLS} applied to $u$ allows us to ensure that~$u(\Omega)\subset B'(\C)$ for some algebraic curve $B'\subset \kX$ all of whose irreducible components dominate $B$.  Making use of the rational simple connectedness of $X$, we construct a comb $Z\subset X$  whose Zariski closure in $\kX$ contains~$B'$ and hence $u(\Omega)$.  The idea is then to deform~$Z$ to a rational curve on $X$, to simultaneously deform~$u$ so that $u(\Omega)$ now belongs to the Zariski closure in $\kX$ of this rational curve, and to conclude by applying the (known) tight approximation property for rational curves
 (Royden's theorem \cite[Theorem~10]{Royden}, or
 see \cite[Theorem~4.2]{bwtight}).
The hardest aspect of the proof (and one that is absent from \cite{Hassett}) is the need to ensure, by means of careful choices of $\kX$,  $B'$ and~$Z$, that the Zariski closures in $\kX$ of the comb, as well as of its small deformations, are smooth over $B$ in a neighborhood of $u(\Omega)$.  This key property is what allows us to deform $u$ with the comb.

\bigskip
The point of view of tight approximation is useful even if one only wants to prove Theorem \ref{th2} for a 
rationally simply connected variety $X$ over $\C(B)$ that comes,  by extension of scalars, from a variety $X_0$ defined over $\C$. Indeed,  in this case,  the model of $X$ on which the proof takes place may not be the constant fibration $X_0\times B\to B$.

 Even though the article \cite{bwtight} also defines and studies the tight approximation property
in the more general context of real algebraic geometry, we are unable to prove a real counterpart of Theorem \ref{th2} (or of its consequence Theorem~\ref{th1}), because many of our arguments break down in this extended setting; for instance, the Graber--Harris--Starr theorem plays a key role in the proof, but no real analogue of it is known.

\subsection*{Organization of the text}

After recalling general facts concerning rationally simply connected varieties in Section \ref{secRSC} and concerning the tight approximation property in Section \ref{sectight},  we give the proof of Theorem \ref{th2} and Corollary \ref{corhyptight} in Section~\ref{secproof}.

\subsection*{Acknowledgements}

We are grateful to Johan de Jong and to Jason Starr for providing useful references to the literature
and in particular for bringing \cite{minoccheri} to our attention.
We thank Finnur Lárusson for pointing us to \cite[Theorem~10]{Royden},
where the tight approximation property
for~$\P^1_\C$ was first proved.  We also thank Franc Forstneri\v{c} for useful comments. We are pleased to acknowledge the influence of \cite{Hassett} on the ideas developed in this article.

\section{Reminders on rationally simply connected varieties}
\label{secRSC}

\subsection{Definition of rationally simply connected varieties}
\label{defRSC}

Let $k$ be a field of characteristic $0$.  An integral projective variety $X$ over $k$ is \textit{rationally connected} if for any algebraically closed field extension $k\subset l$,  any two general points $x,y\in X(l)$ are in the image of a rational curve $\P^1_l\to X_l$ (see \cite[IV, Definition 3.2.2]{Kollarbook}).  Being rationally connected is a birational invariant of integral projective varieties over $k$. 

We will make use of Kontsevich's moduli spaces of stable maps (see \cite{FuPa} for an introductory reference). 
Here is what we will need to know about them.
Let $X$ be a projective variety over $k$.  For $g, n\geq 0$, let $\oM_{g,n}(X)$ be the coarse moduli space of stable maps from projective connected nodal curves of genus~$g$ with~$n$ marked points to $X$ (see e.g.\,\cite[Theorem 50]{AraKo} applied with $S=\Spec(k)$,  $P=\varnothing$, $Q=\Spec(k^n)$ and varying~$d$).
 It is a countable disjoint union (indexed by the integer~$d$ appearing in \emph{loc.\,cit.})\ of projective varieties over $k$.  Its construction as a quotient of a scheme built out of Hilbert schemes \cite[\S 57]{AraKo} implies the following facts. 
\begin{enumerate}
\item Evaluation at the marked points induces a morphism $\ev:\oM_{g,n}(X)\to X^n$.
\item There is an open subset $\oM^{\textrm{fine}}_{g,n}(X)\subset\oM_{g,n}(X)$ corresponding to stable maps with no automorphisms, and this open subset carries a universal stable map. 
\item There is an open subset $M_{g,n}(X)\subset\oM_{g,n}(X)$ corresponding to stable maps from a smooth curve of genus $g$ with $n$ marked points.
\end{enumerate} 

In addition,  as $\oM_{g,n}(X)$ is the coarse moduli space of the moduli stack $\overline{\mathcal{M}}_{g,n}(X)$ of stable maps (a proper Deligne--Mumford stack,  see \cite[Theorem 3.14]{BM}), the following assertion is a consequence of \cite[Theorem 3.6]{BM}.

\begin{enumerate}
\setcounter{enumi}{3}
\item There exists a forgetful morphism $\oM_{g,n}(X)\to\oM_{g,0}(X)$ whose effect on a stable map is to forget the marked points and to stabilize the resulting map.
\end{enumerate} 

We may now give the following definition.

\begin{defn}
\label{defiRSC}
A smooth projective variety $X$ over $k$ is said to be \textit{rationally simply connected} if there exist an irreducible component $M$ of $\oM_{0,2}(X)$ and irreducible components $(M_e)_{e\geq 3}$ of $\oM_{0,0}(X)$ such that the following assertions hold.
\begin{enumerate}[(i)]
\item 
\label{i}
The variety $M$ intersects $M_{0,2}(X)$ and the evaluation morphism $\ev|_M:M\to X^2$ is dominant with rationally connected generic fiber.
\item 
\label{ii}
For all $e\geq 3$, the variety $M_e$ intersects $M_{0,0}(X)$ and is rationally connected.
\item 
\label{iii}
Fix $e\geq 3$ and an algebraically closed field extension $k\subset l$.  Let $x\in \oM_{0,0}(X)(l)$ be a stable map to $X_l$ whose domain is a comb (see \cite[II, Definition~7.7]{Kollarbook}) with $e$ teeth
and a handle isomorphic to $\P^1_l$.
Assume that the handle is contracted in $X_l$ and that each tooth is a free rational curve giving rise,  when marked by its intersection with the handle and  by some other point, to an $l$-point of~$M\subset \oM_{0,2}(X)$.  Then $x\in M_e(l)$.
\end{enumerate}
\end{defn}

\begin{rmk}
The condition $e\geq 3$ in Definition \ref{defiRSC} ensures that the combs considered in (\ref{iii}) satisfy the stability condition of stable maps.
\end{rmk}

\subsection{Examples of rationally simply connected varieties}
\label{exRSC}

Low degree complete intersections of hypersurfaces in projective space constitute
the main source of rationally simply connected varieties to which our
results can be fruitfully applied.

\begin{thm}[de Jong, Starr \cite{djs}]
\label{th:djs}
Let~$k$ be a field of characteristic~$0$.
Let $c \geq 1$, $n \geq 1$, $d_1,\dots,d_c\geq 2$ be integers such that
$n+1 \geq \sum_{i=1}^c d_i^2$
and $n-c \geq 3$.
Any
smooth complete intersection
 $X \subset \P^n_k$
 of hypersurfaces of degrees~$d_1,\dots,d_c$
is rationally simply connected.
\end{thm}

A published proof of this theorem can be found in the work of
Minoccheri~\cite{minoccheri}, who also extends it to complete intersections in weighted projective
spaces.
We now briefly recall which are the components~$M$ and~$(M_e)_{e \geq 3}$ produced
by de~Jong and Starr's arguments,
in order to clarify that \cite{minoccheri} proves Theorem~\ref{th:djs}
(as well as its extension to weighted projective spaces
\cite[Theorem~1.2]{minoccheri})
when ``rationally simply connected''
is understood in the sense of
Definition~\ref{defiRSC} above rather than in the sense
of \cite[Definition~2.9]{minoccheri}.

Let~$X$ be as in Theorem~\ref{th:djs}.
By \cite[\textsection3]{minoccheri},
there exists a unique irreducible component $L \subset \oM_{0,1}(X)$
such that a general point of~$L$ parametrises a pointed line in~$X$ and such that
the morphism $\ev|_L: L \to X$ is dominant.
Moreover, the generic fibre of $\ev|_L$
is then geometrically irreducible.
Let~$L_1$ denote the image of~$L$ by the forgetful morphism $\oM_{0,1}(X) \to \oM_{0,0}(X)$.

The line parametrised by
a general point of~$L_1$ is free (see \cite[II, Proposition~3.10]{Kollarbook}).
Applying \cite[Lemma~3.5, Lemma~3.6]{djs}
(see \cite[Proposition~4.3, Proposition~4.6]{Hassett} for a proof in print) therefore shows
that for all $e\geq 2$, there exists a unique irreducible component $L_e\subset \oM_{0,0}(X)$ 
such that
\begin{enumerate}
\item the curves parametrised by~$L_e$ have degree~$e$ in~$\P^n_k$;
\item a general point of~$L_e$ parametrises a free rational curve on~$X$;
\item for any algebraically closed field extension $k\subset l$,
any free map $g:\P^1_l \to X_l$
and any degree $e$ map $m : \P^1_l \to \P^1_l$,
if~$g$ is parametrised by a point of~$L_1(l)$,
then $g\circ m$ is parametrised by a point of~$L_e(l)$;
\end{enumerate}
 in addition, these components automatically satisfy the following property:
\begin{enumerate}
\setcounter{enumi}{3}
\item\label{enum:compirrlditem4} for any algebraically closed field extension $k\subset l$
and any $x \in \oM_{0,0}(X)(l)$,
if the non-contracted irreducible components of the domain of the stable map~$x$
are free rational curves giving rise to $l$\nobreakdash-points of $L_{e_1}, \dots, L_{e_s}$,
then $x \in L_{e_1 + \dots + e_s}(l)$. 
\end{enumerate}

For all $e \geq 1$, set $M_e:=L_{2e}$  and
let $M_{e,2}$ denote
the unique irreducible component of
$\oM_{0,2}(X)$ that dominates~$M_e$ via the forgetful morphism
$\oM_{0,2}(X) \to \oM_{0,0}(X)$.
Finally, set $M:=M_{1,2}$.
Assertion~(iii) of Definition~\ref{defiRSC}
is a consequence of~\eqref{enum:compirrlditem4} (with all~$e_i=2$).
Let us discuss the other assertions.

According to \cite[Proposition~4.2]{minoccheri},
there exists a unique irreducible component $M' \subset \oM_{0,2}(X)$
parametrising conics on~$X$ such that the morphism $\ev|_{M'}:M'\to X^2$ is dominant.
Applying \cite[Proposition~5.2, Proposition~5.3, Claim on p.~505]{minoccheri},
we see
that two general points of~$X$ can be connected by a chain of two free lines.
Such a chain corresponds, on the one hand, to a point of~$M'$
(when marked by the two general points),
and, on the other hand,
to a point of~$L_2$
(thanks to~\eqref{enum:compirrlditem4},
as any line going through a general point of~$X$  corresponds to a point of~$L_1$).
This point of~$L_2$ does not belong to any other
 irreducible component
of~$\oM_{0,0}(X)$
(see \cite[Theorem~51]{AraKo}).
We deduce that the image of~$M'$ by the forgetful map $\oM_{0,2}(X) \to \oM_{0,0}(X)$
is contained in~$L_2$.
Having the same dimension as~$L_2$
(as a consequence of \cite[Theorem~51]{AraKo}), it is therefore equal to it;
hence $M=M'$.
Thus $\ev|_M$ is dominant, as required by~(i) of Definition~\ref{defiRSC}.

As explained in \cite[pp.~507--508]{minoccheri},
the generic fibre of $\ev|_{M_{e,2}}:M_{e,2} \to X^2$ 
is rationally connected
for all $e \geq 1$.  This implies, on the one hand, the remaining part of
assertion~(i) from Definition~\ref{defiRSC},
and on the other hand, the rational connectedness of~$M_{e,2}$
(by \cite[Corollary~1.3]{ghs}) and hence that of~$M_e$ for all~$e$,
thus completing the proof of assertion~(ii) from Definition~\ref{defiRSC}.

\section{Generalities on the tight approximation property}
\label{sectight}

\subsection{Definition}

Let us first recall the precise definition of tight approximation.

\begin{defn}
\label{deftight1}
Let $f:\kX\to B$ be a flat morphism of smooth projective varieties over~$\C$ with $B$ a connected curve. We say that $f$ satisfies the tight approximation property if for any open subset~$\Omega\subset B(\C)$, any $m\geq 0$,
any $b_1,\dots, b_m\in\Omega$, any~$r\geq 0$ and any holomorphic section $u:\Omega\to \kX(\C)$ of $f:\kX(\C)\to B(\C)$ over~$\Omega$, there is a sequence $(s_n)_{n\geq 0}$ of algebraic sections of $f$ with the same $r$-jets as $u$ at the~$b_i$ such that $(s_n|_{\Omega})_{n\geq 0}$ converges to $u$ for the compact-open topology.
\end{defn}

\begin{rmk}
\label{remweaktopo}
In Definition \ref{deftight1}, it would be equivalent to require that $(s_n|_{\Omega})_{n\geq 0}$ converge to $u$ in the weak $\ci$ topology (as we do in \cite[Definition 2.1]{bwtight}), by the Cauchy estimates (see \cite[Proposition 2.2]{bwtight}).
\end{rmk}

The tight approximation property only depends on the generic fiber of $f$---in fact, only on its birational equivalence class (see \cite[Theorem 3.1]{bwtight}).  This justifies the next definition (see \cite[Definition~3.3]{bwtight}).

\begin{defn}
\label{deftight}
Let $B$ be a connected smooth projective curve over $\C$.  A smooth variety $X$ over $\C(B)$ satisfies the tight approximation property if so does one (equivalently, any) of its projective regular models $f:\kX\to B$.
\end{defn}

When the points $b_1,\dots, b_m$ and the integer $r$ are clear from the context, we will informally say that two holomorphic sections $u_1$ and $u_2$ of $f$ over $\Omega$ are close in the sense of tight approximation if they have the same $r$-jets at the $b_i$ and if they are close for the compact-open topology.

\begin{rmks}
\label{remOka1}
(i) For smooth projective varieties defined over $\C$, the tight approximation property is closely related to the \textit{algebraically Oka-1 property} introduced by Forstneri\v{c} and L\'arusson in \cite{FL}. More precisely, if $X$ is a smooth projective variety over~$\C$, then~$X$ is algebraically Oka-1 in the sense of \cite[Definition~1.5]{FL} if and only if~$X_{\C(B)}$ satisfies the tight approximation property for all connected smooth projective curves~$B$ over $\C$. This equivalence results at once from the definitions and from the following three facts.

\begin{enumerate}[(a)]

\item In Definition \ref{deftight1}, one can assume that $\Omega\subsetneq B(\C)$. Indeed, if $\Omega=B(\C)$, then  the holomorphic section $u:\Omega\to B(\C)$ is already algebraic, by GAGA.

\item In \cite[Definition~1.5]{FL} (and using the notation introduced there),  when the variety $X$ is projective, the hypothesis that $K$ be Runge in $R$ is superfluous.  The argument appears in the proof of \cite[Corollary 1.8]{FL}. 
One reduces to the case where $K$ is Runge by enlarging~$K$ to ensure that it is smoothly bounded, and by removing from $R$ one point in each connected component of $R\setminus K$.

\item If $X_{\C(B)}$ satisfies the tight approximation property (in the sense of Definition~\ref{deftight}) for some connected smooth projective curve $B$ over $\C$, then $X$ is rationally connected, by
\cite[Corollary~2.16]{Hassett} and \cite[Remark 4.4 (4)]{Debarre}, and hence simply connected (see \cite[Corollary 4.29]{Debarre}). As a consequence, the condition that $f$ and $F$ be homotopic in \cite[Definition 1.5]{FL} is automatically satisfied in this case, by obstruction theory \cite[37.12]{Steenrod}, as an affine algebraic Riemann surface has no cohomology in degree $\geq 2$.
\end{enumerate}

(ii)
We recall that any smooth projective variety over~$\C$ that is algebraically Oka-1
in the sense of \cite[Definition~1.5]{FL}
is also \emph{Oka-1} in the sense of \cite[Definition~1.1]{AF} (see \cite[Proposition~1.9]{FL}).
In contrast with Remark~\ref{remOka1}~(i), there exist smooth projective varieties $X$ over $\C$ that are
Oka-1 but that are not rationally connected (e.g.\ elliptic curves, see \cite[Corollary 2.6]{AF}), and hence such that~$X_{\C(B)}$ does not satisfy the tight approximation property for any connected smooth projective curve~$B$ over $\C$.
\end{rmks}

\subsection{Perturbing multisections}

The following proposition will allow us to perturb a multisection of $f:\kX\to B$ in the sense of tight approximation. We let $\Res_{k'/k}$ denote Weil restriction of scalars
 from~$k'$ to~$k$.

\begin{prop}
\label{lemperturb}
Let $f:\kX\to B$ and $\pi:B'\to B$ be dominant morphisms of smooth connected projective varieties over~$\C$.  Assume that $B$ and $B'$ are curves
and that the generic fiber $X$ of $f$ is rationally connected.
Set $k=\C(B)$ and $k'=\C(B')$.
  Let $U\subset \Res_{k'/k}(X_{k'})$ be a dense open subset
and $K \subsetneq B'(\C)$ be a compact subset.
Fix~$m \geq 0$, $r \geq 0$ and $b_1,\dots, b_m\in K$.
For any morphism $g:B'\to\kX$ with~$f\circ g=\pi$,
there exists
 a morphism $\hat{g}:B'\to\kX$ with $f\circ \hat{g}=\pi$, with the same $r$-jets as $g$ at the~$b_i$, such that $\hat{g}|_K$ is arbitrarily close to $g|_K$ for the compact-open topology, and such that the generic point of $\hat{g}(B')$ belongs to $U(k)\subset \Res_{k'/k}(X_{k'})(k)=X(k')$.
\end{prop}

\begin{proof}
After replacing~$\kX$ with~$\kX \times \P^3_\C$ and composing~$f$ with the first projection morphism
$\kX \times \P^3_\C \to \kX$, we may assume that $\dim(X)\geq 3$.

Fix a desingularisation $\nu:\kX' \to \kX \times_B B'$ of~$\kX \times_B B'$
that induces an isomorphism over the generic point of~$B$ (using~\cite{Hironaka}),
and let $f':\kX' \to B'$ denote its composition
with the base change of~$f$.
Fix the morphism~$g$.
The strict transform $g' :B' \to \kX'$ of $(g,\Id):B' \to \kX \times_B B'$
is then a section of~$f'$.

Let $\beta \subset B'$ denote the finite subscheme whose underlying set is $\{b_1,\dots,b_m\}$
and each of whose connected components has degree~$r+1$ over~$\C$
(so that the collection of the $r$\nobreakdash-jets of~$g$ at the~$b_i$ can be identified with $g|_\beta$).

Let $F = \Res_{k'/k}(X_{k'}) \setminus U$.  Let $\kF \subset \Res_{B'/B}(\kX')$ denote the Zariski closure
of~$F$.
The fibre~$\kF_b$ of $\kF \to B$ above a general point $b \in B(\C)$ is a strict closed subset
of the smooth projective rationally connected variety $\prod_{b' \in \pi^{-1}(b)} \kX_b$,
where
$\kX_b=f^{-1}(b)$.
We henceforth fix such a point $b \in B(\C)$, distinct from the $\pi(b_i)$.

We would like to deform~$g'$ to a $\hat g'$ such that $(\hat g'(b'))_{b' \in \pi^{-1}(b)} \notin \kF_b$.
There might be obstructions to this deformation problem.
To bypass them, we employ a
technique introduced by Graber, Harris and Starr in~\cite[\textsection2.1]{ghs}
(see also~\cite[\textsection1]{Kollarspecialization}),
which consists in attaching to the curve~$g'(B')$ sufficiently many
free rational curves
contained in the fibres of~$f'$, at general points
of~$g'(B')$, along general tangent directions,
to obtain a comb that can then be deformed to the desired section.
To be precise, applying
\cite[Remark~9 and Proposition~12]{Kollarspecialization}
as in the proof of \cite[Theorem~16]{Kollarspecialization},
and noting that $\dim(X) \geq 3$,
we see that for any infinite subset $P \subset B'(\C)$,
there exist an integer~$d$, pairwise distinct points $p_1,\dots,p_d \in P$,
smooth rational curves $T_i \subset f'^{-1}(p_i)$ for $i \in \{1,\dots,d\}$,
such that the local complete intersection curve $C=g'(B') \cup T_1 \cup \dots \cup T_d \subset \kX'$ is
connected and such that the following two conditions hold,
where for any
closed subscheme $S \subset B'$, we denote by~$I_{S}$ the ideal sheaf
of~$S$ viewed as a closed subscheme of~$C$ via~$g'$:
\begin{enumerate}
\item
\label{cond:vbgg}
the vector bundle $N_{C/\kX'} \otimes_{\sO_C} I_\beta$ on~$C$ is generated by its global sections;
\item
\label{cond:h1cnivanishes}
the group $H^1(C,N_{C/\kX'} \otimes_{\sO_C} I_{\pi^{-1}(b) \cup \beta })$ vanishes.
\end{enumerate}
We apply this assertion with $P = B'(\C)\setminus (K \cup \pi^{-1}(b))$.

Let $L \subset \Hilb_{\kX'/\C}$ be the Hilbert scheme parametrising the closed subschemes of~$\kX'$
that contain the (non-reduced) finite subscheme $g'(\beta) \subset \kX'$.
We recall that the tangent space $T_{[C]}L$ to~$L$ at the point $[C] \in L(\C)$ corresponding to~$C$ can be identified
with
$H^0(C,N_{C/\kX'} \otimes_{\sO_C} I_{ \beta })$,
and that the vanishing of
$H^1(C,N_{C/\kX'} \otimes_{\sO_C} I_{\beta })$,
which follows from condition~\eqref{cond:h1cnivanishes}, forces~$L$ to be smooth at~$[C]$.
These two facts result, for example, from \cite[Proposition~3.14]{starrarithmeticoverfunctionfields}
and \cite[Proposition~4.5.2]{sernesi}.

Let $L^0 \subset L$ be a smooth connected open neighbourhood of~$[C]$.
Let $W \subset L^0 \times \kX'$ be the tautological closed subscheme flat over~$L^0$.

\begin{lem}
\label{lem:generalpointofl0}
The fibre of $W \to L^0$ above a general point of~$L^0(\C)$
is a section of~$f'$.  When viewed as an element of~$X(k')$, it belongs to the subset~$U(k)$.
\end{lem}

\begin{proof}
First-order deformations of~$C$ as a closed subscheme of~$\kX'$ that contains~$g'(\beta)$ are parametrised
by $T_{[C]}L^0=H^0(C,N_{C/\kX'} \otimes_{\sO_C} I_{ \beta })$.
Condition~\eqref{cond:vbgg} ensures that
the
 first-order deformations of~$C$
which do not smooth a given node of~$C$ form a linear subspace of~$T_{[C]}L^0$ of positive codimension.
Thus, a general first-order deformation of~$C$ smoothes all of the nodes of~$C$.  It follows that
the generic fibre of $W \to L^0$
is smooth.  As it is connected
(by \cite[Théorème~12.2.4~(vi)]{EGA43}) and as it has degree~$1$ over~$B'$, it must be a section of~$f'$.

It remains to verify the second assertion of Lemma~\ref{lem:generalpointofl0}.
By the conclusion of the previous paragraph, the 
restriction $\rho:W \to L^0 \times B'$  of $(\Id,f') : L^0 \times \kX' \to L^0 \times B'$
to~$W$ is a birational morphism---it even induces an isomorphism over the generic point of~$L^0$.
As~$\rho$ is proper
and as its fibres over $\{[C]\} \times \pi^{-1}(b)$
are finite,
there exists
an open neighbourhood $V \subset L^0 \times B'$
of   $\{[C]\} \times \pi^{-1}(b)$
such that the morphism  $\rho:\rho^{-1}(V)\to V$ is finite
(see \cite[Corollaire~4.4.11]{EGA31})
and is therefore an isomorphism
(see \cite[Lemme~8.12.10.1]{EGA43}).

The subset $L^{1} \subset L^0$ of all $x \in L^0$ such that $(x,b') \in V$
for all $b' \in \pi^{-1}(b)$ is an open neighbourhood of~$[C]$.  After
replacing~$L^0$ with~$L^{1}$ and~$V$ with $V \cap (L^{1} \times B')$, we
may assume that $L^0 \times \pi^{-1}(b) \subset V$.
Composing this inclusion
with  $\rho^{-1} : V \isoto \rho^{-1}(V) \subset W$ and with the projection $W\to \kX'$
yields a morphism $L^0 \times \pi^{-1}(b) \to \kX_b$ (recall that $\kX_b=\kX'_{b'}$ for $b'\in \pi^{-1}(b)$), and hence
a morphism $\varphi:L^0 \to \prod_{b' \in \pi^{-1}(b)} \kX_b$.

Setting $\sF=N_{C/\kX'} \otimes_{\sO_C} I_{ \beta }$,
the differential $T_{[C]}L^0 \to \prod_{b' \in \pi^{-1}(b)} T_{g(b')}\kX_b$
of~$\varphi$ at~$[C]$ can be identified with the restriction map
$H^0(C,\sF) \to H^0(\pi^{-1}(b), \sF|_{\pi^{-1}(b)})$,
which, by
condition~\eqref{cond:h1cnivanishes}, is surjective.
As~$\varphi$ is a morphism between smooth varieties, we deduce that~$\varphi$ is smooth
at~$[C]$, hence that~$\varphi$ is dominant
and, in particular, that~$\varphi(L^0) \not\subset \kF_b$.
The second assertion of Lemma~\ref{lem:generalpointofl0} follows.
\end{proof}

Let~$J$ be a finite set. For each $j \in J$, let $K'_j \subseteq K$ and $\Xi_j \subseteq \kX'(\C)$
be a compact and an open subset, respectively,
such that~$g'$ maps~$K'_j$ into~$\Xi_j$.
The complement of the image of
\begin{align*}
W(\C) \cap \big(L^0(\C) \times \bigcup_{j \in J} \big(f'^{-1}(K'_j) \cap (\kX'(\C) \setminus \Xi_j)\big)\big)
\end{align*}
by the projection $W(\C) \to L^0(\C)$ 
is an open subset of~$L^0(\C)$
which is nonempty as it contains~$[C]$; it is therefore Zariski dense.
In view of
Lemma~\ref{lem:generalpointofl0},
the fibre of $W \to L^0$ above a Zariski general point of this open subset is a section of~$f'$
 that maps~$K'_j$ into~$\Xi_j$ for all $j \in J$,
that has the same
$r$\nobreakdash-jets as~$g'$ at the~$b_i$,
and that gives rise to a point of~$U(k)$.
Composing it with the projection $\kX' \to \kX$
 yields the desired  morphism~$\hat g$.
\end{proof}

\subsection{Nash approximation}

The next proposition, which is an application of the main theorem of Demailly, Lempert and
Shiffman~\cite{DLS}, states that a holomorphic section $u:\Omega\to\kX(\C)$
as above is always close to a Nash section (that is, to a holomorphic section taking its values in an algebraic curve) in the sense of tight
approximation.

\begin{prop}
\label{propDLS}
Let $f:\kX\to B$ be a surjective
flat morphism of smooth projective varieties over~$\C$ with $B$ a connected curve.
Let $\Omega' \subset \Omega \subsetneq B(\C)$
be open subsets such that~$\Omega'$ is  relatively compact in~$\Omega$.
For any $m\geq 0$,
any $b_1,\dots, b_m\in\Omega'$, any~$r\geq 0$ and any holomorphic section $u:\Omega\to \kX(\C)$ of $f:\kX(\C)\to B(\C)$ over~$\Omega$, there is a sequence $(\hat u_n)_{n\geq 0}$ of holomorphic sections of $f$ over~$\Omega'$,
with the same $r$-jets as $u$ at the~$b_i$, such that
 $(\hat u_n)_{n\geq 0}$
 converges to~$u|_{\Omega'}$ for the
compact-open topology and such that for all $n \geq 0$,
the Zariski closure of
$\hat{u}_n(\Omega')$ in~$\kX$ has dimension~$1$.
\end{prop}

\begin{proof}
Fix~$m$, $b_1,\dots,b_m$, $r$ and~$u$.
Let~$K$ denote the closure of~$\Omega'$ in~$\Omega$.
Choose a relatively compact open neighbourhood~$\Omega'' \subset \Omega$ of the compact subset $K \subset \Omega$.

As~$\Omega$ is Stein, being an open Riemann surface
(Behnke and Stein~\cite{BS}, see also \cite[V, \textsection1.5]{Stein}),
we can apply \cite[Corollary~1.7]{DLS}
and obtain
a sequence~$(\hat v_n)_{n \geq 0}$ of holomorphic maps $\hat v_n:\Omega'' \to \kX(\C)$,
with the same $r$\nobreakdash-jets as~$u$ at the~$b_i$, such that~$(\hat v_n)_{n \geq 0}$ converges
to~$u|_{\Omega''}$ for the compact-open topology and such that for all~$n \geq 0$, the Zariski closure
of~$\hat v_n(\Omega'')$ in~$\kX$ has dimension~$1$.

The map~$\hat v_n$ need not be a section of~$f$.
To correct this and thus complete the proof of Proposition~\ref{propDLS},
it suffices to verify that for $n\gg 0$, the
map $(f \circ \hat v_n)^{-1}(\Omega') \to \Omega'$
obtained by restricting $f \circ \hat v_n : \Omega'' \to B(\C)$
is biholomorphic.  Composing~$\hat v_n$
with the inverse of this map will then yield the desired holomorphic sections~$\hat u_n$ of~$f$.

The sequence $(f \circ \hat v_n)_{n\geq 0}$ converges to the inclusion $\Omega'' \subset B(\C)$ for the compact-open topology,
and therefore also for the weak~$\sC^1$ topology (see Remark \ref{remweaktopo}).
Choosing a compact neighbourhood~$K'$ of~$K$ in~$\Omega''$
with~$\sC^1$ boundary
and applying \cite[Chapter~2, Theorem~1.4]{Hirsch} to the
inclusion $K' \subset B(\C)$,
we see that the maps~$(f \circ \hat v_n)|_{K'}$  are injective for $n \gg 0$. 
After shrinking~$\Omega''$, we may therefore
assume that $f \circ \hat v_n$ is injective for  $n \gg 0$.
The map $f \circ \hat v_n$ then induces a biholomorphic map
between~$\Omega''$ and the open subset 
 $(f \circ \hat v_n)(\Omega'') \subset B(\C)$ for $n \gg 0$
(see \cite[Chapter~1, Corollary~2.4 and Corollary~2.5]{forster}).

It remains to check that $\Omega' \subset (f \circ \hat v_n)(\Omega'')$
for $n \gg 0$.
As~$K$ is compact and contains~$\Omega'$ and as $(f \circ \hat v_n)(\Omega'')$ is open,
it suffices to see that any  $b \in K$
belongs to~$(f \circ \hat v_n)(\Omega'')$ when~$n$ is larger than a bound that can depend on~$b$.
After choosing a local chart for~$B(\C)$ around~$b$,  we are reduced to
remarking that if $\iota:\Dbar \hookrightarrow \C$ denotes the inclusion of the closed unit disk
and $\delta:\Dbar \to \C$ is continuous and satisfies $\delta(\Dbar)\subset \Dbar$, then $0 \in (\iota+\delta)(\Dbar)$,
by Brouwer's fixed-point theorem applied to~$-\delta$.
\end{proof}

\section{Proof of the main theorem}
\label{secproof}

In this section we prove Theorem~\ref{th2} and Corollary~\ref{corhyptight}.

\begin{proof}[Proof of Theorem~\ref{th2}]
We proceed in several steps.

\begin{Step}
\label{stepmodel}
Fixing a model of $X$ and some (multi)sections of it.
\end{Step}

Define $k:=\C(B)$. Let $X$ be a smooth, projective, rationally simply connected variety over $k$.
If~$\dim(X)=1$, then $X$ is a conic, and hence is isomorphic to $\P^1_k$ by Tsen's theorem.  It therefore satisfies the tight approximation property by \cite[Theorem 4.2]{bwtight}. We henceforth assume that $\dim(X)\geq 2$.

Fix a projective regular model $f:\kX\to B$ of $X$.  Choose an open subset $\Omega\subset B(\C)$,  points $b_1,\dots, b_m\in \Omega$, an integer $r\geq 0$ and a holomorphic section $u:\Omega\to\kX(\C)$ of~$f:\kX(\C)\to B(\C)$ over $\Omega$ (as in Definition~\ref{deftight1}).  To prove Theorem \ref{th2}, we must construct a section of $f$ that has the same $r$-jets as $u$ at the $b_i$ and whose restriction to $\Omega$ is close to $u$ for the compact-open topology.

If $\Omega=B(\C)$, then $u$ is itself algebraic by GAGA. Consequently, we
may assume that~$\Omega\subsetneq B(\C)$.
By Proposition \ref{propDLS}, we may then assume, after replacing $\Omega$
with a smaller open subset containing the given compact subsets over which
we want to approximate~$u$, and after perturbing~$u$, that there exists a
possibly reducible algebraic curve $B'\subset\kX$ such that~$u(\Omega)\subset B'(\C)$.

After adding irreducible components to $B'$, we may assume that the degree $d$ of $B'$ over~$B$ is $\geq 3$.
After discarding the vertical components of $B'$, we may assume that the irreducible components of $B'$ all dominate $B$.  Let $(B'_j)_{j\in J}$ be their normalizations and let~${\pi_j:B'_j\to B}$ and~${g_j:B'_j\to\kX}$ be the induced morphisms. 
The field~$k_j:=\C(B'_j)$ is a finite extension of~$k$. We set $d_j:=[k_j:k]$.
Let $x_j\in X(k_j)$ be the generic point of $g_j(B'_j)$.
We also let $\tilde{u}:\Omega\to \sqcup_{j\in J} B'_j(\C)$ denote the strict transform of~$u$.

By the Graber--Harris--Starr theorem \cite[Theorem 1.1]{ghs}, the morphism $f$ admits a section $s:B\to \kX$. We
may assume that~$s(B)\not\subset B'$ (see \cite[IV, Theorem~6.10]{Kollarbook}).
Let~$x\in X(k)$ be the generic point of $s(B)$. 

\begin{Step}
\label{generalmult}
Perturbing the (multi)sections to put them in general position.
\end{Step}

Applying the hypothesis that $X$ is rationally simply connected provides us with projective varieties $M$ and $(M_e)_{e\geq 3}$ over $k$ (see Definition \ref{defiRSC}).  It is a consequence of Definition \ref{defiRSC} (\ref{i}) and of \cite[IV, Theorem 3.11]{Kollarbook} that there exists a dense open subset $V\subset X^2$ over which the fibers of $\ev|_M:M\to X^2$ are all rationally connected. (To be precise, let $\nu:\widetilde{M}\to M$ be a resolution of singularities and apply \cite[IV, Theorem 3.11]{Kollarbook} to the restriction of $\ev|_M\circ\nu$ over
a dense open subset of~$X^2$.)

The variety $M$ generically parametrizes very free rational curves on $X$ (combine Definition \ref{defiRSC} (\ref{i}) and \cite[II, Corollary~3.10.1]{Kollarbook}).  Using \cite[II, Theorem~3.14]{Kollarbook},
 one deduces that the rational curves generically parametrized by $M$ are generically embedded. After shrinking $V$, we may therefore assume that any fiber of~${\ev|_M:M\to X^2}$ over a point of $V$ generically parametrizes generically embedded very free rational curves on~$X$.

Let $\ok$ be an algebraic closure of $k$.
Fix $\ok$-isomorphisms $k_j\otimes_k\ok\isoto \ok^{d_j}$ for $j\in J$.  Consider the smooth projective algebraic variety $Y:=X\times \prod_{j\in J}\Res_{k_j/k}(X_{k_j})$ over~$k$. One has $Y(\ok)=X(\ok)\times\prod_{j\in J}X_{k_j}(k_j\otimes_k\ok)=X(\ok)\times\prod_{j\in J}X(\ok)^{d_j}$. Let $U\subset Y$ be a dense open subset whose $\ok$-points are tuples
$\big(y,(y_{j,1},\dots,y_{j,d_j}\mkern-1mu)_{j\in J}\big)\in Y(\ok)$ satisfying
the following property: for all $j\in J$ and all $1\leq i\leq d_j$, one has $(y,y_{j,i})\in V(\ok)$, and a general $\ok$-point of the fiber of $\ev|_M:M\to X^2$ above $(y,y_{j,i})$ parametrizes a generically embedded rational curve on $X_{\ok}$ that avoids $y_{j',i'}$ for all $(j',i')\neq (j,i)$. Such a dense open subset exists because $\dim(X)\geq 2$.

View $(x,(x_j)_{j\in J})\in X(k)\times\prod_{j\in J}X(k_j)$ as an element of $Y(k)$.  Perturbing the morphisms $s:B\to\kX$ and $(g_j:B'_j\to\kX)_{j\in J}$ by means of Proposition \ref{lemperturb} (and changing the points $x$ and~$(x_j)_{j\in J}$ and the holomorphic map $u$ accordingly, without modifying the~$(\pi_j)_{j\in J}$ and~$\tilde{u}$), we may assume that $(x,(x_j)_{j\in J})\in U(k)$.

\begin{Step}
\label{stepteeth}
Choosing the teeth of a future comb on $X$.
\end{Step}

As $(x,(x_j)_{j\in J})\in U(k)$ (see Step \ref{generalmult}), one has $(x,x_j)\in V(k_j)$.
Let $F_j$ be the variety over~$k_j$ that is the fiber of $\ev|_M$ over $(x,x_j)\in V(k_j)$.
Let $F^0_j\subset F_j$ be any dense open subset consisting of smooth points parametrizing generically embedded very free rational curves on~$X$.
As $F_j$ is rationally connected,  the Graber--Harris--Starr theorem \cite[Theorem~1.2]{ghs} and \cite[IV, Theorem 6.10]{Kollarbook} (both applied to a resolution of singularities of $F_j$) together
show that $F^0_j(k_j)\neq\varnothing$. 

 Fix a point $w_j\in F^0_j(k_j)$.  For $j\in J$,  let $(w_{j,1},\dots, w_{j,d_j})\in F^0_j(\ok)^{d_j}$ denote the image of $w_j\in F^0_j(k_j)$ by the bijection $F^0_j(k_j\otimes_k\ok)\isoto F^0_j(\ok)^{d_j}$. The hypothesis that~$(x,(x_j)_{j\in J})\in U(k)$ also implies that, if the $F^0_j$ have been chosen small enough,  the rational curves on~$X_{\ok}$ parametrized by the $(w_{j,i})_{j\in J,1\leq i\leq d_j}$ are all distinct.

As $F^0_j$ parametrizes generically embedded rational curves, one has $F^0_j\subset \overline{M}^{\textrm{fine}}_{0,2}(X)_{k_j}$. 
It follows that $w_j\in F^0_j(k_j)$ corresponds to a very free
morphism $\alpha_j:\P^1_{k_j}\to X_{k_j}$ that satisfies $\alpha_j(0)=x$
and $\alpha_j(\infty)=x_j$ and is a generic embedding.
Since the~$\alpha_j$, for varying $j \in J$, give rise to
 $\sum_{j\in J}d_j$ distinct rational curves on~$X_\ok$,
 the induced morphism $\alpha:=\sqcup_{j\in J}\alpha_j:\sqcup_{j\in J}\P^1_{k_j}\to X$ is also  an
 embedding away from finitely many points. 

We finally ensure that the morphism $\alpha$ is an embedding over an open neighborhood of $x_j$,  for all $j\in J$. To do so, we do not modify the morphisms $\alpha_j$. Instead, we perturb the point $\infty\in\P^1(k_j)$ by applying Proposition~\ref{lemperturb} to the second projection morphism $\P^1_{\C}\times B'_j\to B'_j$ and to its section given by the point at infinity.  We replace the point~$\infty\in\P^1(k_j)$ by the generic point of the section produced by Proposition~\ref{lemperturb} (but still call it $\infty$ after a reparametrization of~$\P^1_{k_j}$) and the point $x_j$ with its image by~$\alpha_j$. We also modify the morphism $g_j:B'_j\to\kX$ and the holomorphic map~$u:\Omega\to\kX(\C)$ accordingly (without changing~$\pi_j$ and~$\tilde{u}$).

\begin{Step}
\label{stepcomb}
Constructing a comb on $X$ and studying its generic deformation.
\end{Step}

Fix a closed embedding $\iota:\sqcup_{j\in J}\Spec(k_j)\hookrightarrow\P^1_k$.  Gluing~$\sqcup_{j\in J}\P^1_{k_j}$ (the teeth) and~$\P^1_k$ (the handle) transversally by attaching $0\in \P^1(k_j)$ (for $j\in J$) to~$\P^1_k$ thanks to $\iota$ produces a comb $C$ over $k$.  Let $\gamma:C\to X$ be the map contracting the handle to~$x\in X(k)$ and equal to $\alpha$ on the teeth. In view of Definition~\ref{defiRSC}~(\ref{iii}), the point $[\gamma]\in \oM_{0,0}(X)(k)$ belongs to $M_d(k)$ (the condition~$d\geq 3$ was enforced in Step~\ref{stepmodel}).

Since the stable map $\gamma$ has no automorphism (as $d\geq 3$ and the teeth are distinct and generically embedded in $X$), the point $[\gamma]\in \oM_{0,0}(X)(k)$ belongs to the open subset $\oM^{\textrm{fine}}_{0,0}(X)\subset \oM_{0,0}(X)$.
As the teeth of the comb are all very free (see Step \ref{stepteeth}), one computes that the group $H^1(C,\gamma^*T_X)$ vanishes. 
It therefore follows from \cite[Theorem~42]{AraKo} that $M_d$ is smooth at $\gamma$.

Let $(t_1,\dots,t_m)$ be a regular system of parameters of $\sO_{M_d,[\gamma]}$. Sending $t_1,\dots,t_m$ to~$m$ algebraically independent elements of $k[[t]]$ without constant terms yields an inclusion of $k$-algebras $\sO_{M_d,[\gamma]}\subset k[[t]]$.  Endow $k(M_d)$ with the $t$-adic valuation. Let~$\sO_{M_d,[\gamma]}\subset R\subset k(M_d)$ be the associated valuation ring with residue field $k$. The induced morphism $\xi:\Spec(R)\to M_d$ sends the closed point of $\Spec(R)$ to~$[\gamma]$.  

As~$[\gamma]\in \oM^{\textrm{fine}}_{0,0}(X)$, the morphism $\xi$ corresponds to a stable map $\delta:D\to X_{R}$ over~$\Spec(R)$.  The restriction of $\delta$ above the closed point of~$\Spec(R)$ is the comb $\gamma:C\to X$.  
As the restriction of $\gamma$ over a neighborhood of the $(x_j)_{j\in J}$ in~$X$
is a closed embedding, it follows from \cite[Proposition~18.12.7]{EGA44} that there is an open subset $W\subset X_R$ containing the points~$(x_j)_{j\in J}$ of the special fiber such that~$\delta|_{\delta^{-1}(W)}:\delta^{-1}(W)\to W$ is a closed embedding.  The restriction of~$\delta$ above the generic point of~$\Spec(R)$, which is a rational curve by Definition~\ref{defiRSC}~(\ref{ii}),  is therefore generically embedded.

As the reduced closed subscheme $\delta(D)\subset X_R$ is flat over $\Spec(R)$, it yields a morphism~$\zeta:\Spec(R)\to\Hilb_{X/k}$. The element  $[\zeta]\in\Hilb_{X/k}(R)\subset \Hilb_{X/k}(k(M_d))$ gives rise to a rational map ${\psi:M_d\dashrightarrow\Hilb_{X/k}}$.  As the components of $\Hilb_{X/k}$ are proper, there exists a resolution of singularities $\mu:\wM_d\to M_d$ such that $\psi$ lifts to a morphism~${\wpsi:\wM_d\to \Hilb_{X/k}}$.  Let~$\wxi:\Spec(R)\to \wM_d$ be the lift of $\xi$ and let~$p\in\wM_d(k)$ be the image by $\wxi$ of the closed point of $\Spec(R)$.  By construction, one has $\mu(p)=[\gamma]$ and $\wpsi\circ\wxi=\zeta$. 

The closed subscheme $Z\subset X$ parametrized by $\wpsi(p)\in\Hilb_{X/k}(k)$ is the special fiber of~$\delta(D)$.  It follows that $Z=\gamma(C)$ as subsets of~$X$ and that $Z\cap W=\gamma(C)\cap W$ as
subschemes of~$X$ (so that $Z$ is smooth at the~$(x_j)_{j\in J}$). We note that~$Z$ need not be reduced outside of~$W$.

\begin{Step}
\label{stepmodif}
Modifying our model of $X$.
\end{Step}

Let $\kZ\subset\kX$ denote the scheme-theoretic closure of $Z\subset X$.
Let $\Sigma \subset \kZ$ denote the closure of the singular locus of~$Z$.
We noted, at the end of Step~\ref{stepcomb}, that the surface~$\kZ$ is smooth at
the points~$(x_j)_{j\in J}$; hence $\Sigma \cap B'$ is finite.

We may legitimately replace~$\kX$ by its blow-up along any closed subscheme of~$\kX$ that does not
dominate~$B$---applying resolution of singularities afterwards if the blow-up center was not smooth---and
$u$, $B'$, $\kZ$, $\Sigma$ and the $(g_j)_{j\in J}$ by their strict transforms in the resulting
modification (without changing $\tilde{u}$ and the $(\pi_j)_{j\in J}$).

Using this procedure,
by repeatedly blowing up the intersection points of~$\Sigma$ with~$B'$ and the singular points
of~$B'$, we ensure that
$\Sigma \cap B'=\emptyset$ and that~$B'$ is smooth.

As the generic fibre of $\kZ \setminus \Sigma\to B$ is smooth,
there exists a closed subscheme $\kY \subset \kZ \setminus \Sigma$ that does not dominate~$B$
and is
such that the blowing up of $\kZ \setminus \Sigma$ along~$\kY$
is smooth. (This follows from \cite{Hironaka};
for a more convenient reference, see \cite[Theorem~1.2.1]{TemkinFunctnonemb} and
\cite[top of~p.133]{Hironaka}.)
Applying the above procedure again,
we blow up~$\kX$ along the scheme-theoretic closure~$\overline{\kY}$ of~$\kY$ in~$\kX$.
As the strict transform of~$\kZ$ coincides with its  blow-up along~$\overline{\kY}$
(see \cite[top of p.~167]{Hironaka}
or \cite[\href{https://stacks.math.columbia.edu/tag/080E}{Lemma 080E}]{stacks-project}),
we ensure, in this way, that~$\kZ$ is smooth along $B'\subset\kZ$.

It follows, in particular, that the morphism $f|_{\kZ}:\kZ\to B$ is itself smooth along the image
of $u:\Omega\to B'(\C)\subset\kX(\C)$.

\begin{Step}
Applying the tight approximation property to a deformation of the comb.
\end{Step}

The closed subscheme~${\kZ\subset\kX}$ is flat over~$B$.
Let $\sigma:B\to\Hilb_{\kX/B}$ denote  the corresponding section.
Let~${\kU\subset \Hilb_{\kX/B}\times_B\kX}$ be the universal closed subscheme.
Let $h_1:\kU\to\Hilb_{\kX/B}$ and~$h_2:\kU\to\kX$ be the projections.
As $h_1$ is flat and $f|_{\kZ}:\kZ\to B$ is smooth along the image of~$u$ (see Step \ref{stepmodif}), the morphism ${h_1:\kU\to\Hilb_{\kX/B}}$ is in fact smooth  along the subset~$\{(\sigma(b),u(b)), b\in\Omega\}$ of $\kU(\C)$.
We deduce from \cite[Proposition 1.3]{FP} (whose proof works in the case of possibly singular complex spaces, as indicated in \emph{loc.\,cit.}) the existence of an open neighborhood $\Theta$ of $\sigma(\Omega)$ in $\Hilb_{\kX/B}(\C)$ and of a holomorphic section~$\tau:\Theta\to\kU(\C)$ of $h_1$ over $\Theta$ whose image is included in the smooth locus of $h_1$, and such that $\tau\circ \sigma(b)=(\sigma(b),u(b))$ for all $b\in \Omega$.

Let us now  take up the notation~$p$, $\mu$, $\wpsi$ introduced in Step~\ref{stepcomb}.

Let $\wM^0_d\subset\wM_d$ be a dense open subset whose image in $M_d$ is included in~$\oM^{\textrm{fine}}_{0,0}(X)$, that parametrizes generically embedded rational curves, and on which the morphism~$\wpsi$ is just given by taking the image of this rational curve; the existence of such an open subset follows from Step \ref{stepcomb}.
As $\wM_d$ is rationally connected (see Definition~\ref{defiRSC}~(\ref{ii})), one can apply Proposition~\ref{lemperturb} to find a point $\hat{p}\in \wM^0_d(k)$  arbitrarily close to~$p$ in the sense of tight approximation.

The point $\mu(\hat{p})\in M_d(k)$ parametrizes a rational curve over $k$ (isomorphic to~$\P^1_k$, by Tsen's theorem) generically embedded into $X$ by a morphism ${\varepsilon:\P^1_k\to X}$, and ${\wpsi(\hat{p})\in\Hilb_{X/k}(k)}$ is the point associated with the reduced closed sub\-scheme~$\varepsilon(\P^1_k)$ of $X$. 
Let $\sigma:B\to\Hilb_{\kX/B}$ and $\hat{\sigma}:B\to\Hilb_{\kX/B}$ be the sections  of $\Hilb_{\kX/B}\to B$ corresponding to $p$ and $\hat{p}$.
 As $p$ and $\hat{p}$ are close in the sense of tight approximation,  the sections $\sigma$ and $\hat{\sigma}$
have the same $r$-jets at the $b_i$, and $\sigma|_{\Omega}$ and $\hat{\sigma}|_{\Omega}$ are close for the compact-open topology. In particular, if $\hat{p}$ was chosen close enough to $p$,
we may assume, after
replacing $\Omega$
with a smaller open subset containing the given compact subsets over which
we want to approximate~$u$,
that $\hat{\sigma}(\Omega)\subset\Theta$.

Define $\hat{u}:\Omega\to \kX(\C)$ by the formula $\hat{u}:=h_2\circ\tau\circ\hat{\sigma}$. As $u=h_2\circ\tau\circ\sigma$, the sections~$\hat{u}$ and $u$ of $f:\kX(\C)\to B(\C)$ over $\Omega$ have the same $r$-jets at the $b_i$ and are close for the compact-open topology. In addition, the image of the section $\hat{u}$ belongs to the smooth locus of the subscheme $\hat{\kZ}\subset \kX$ associated with $\hat{\sigma}$. The generic fiber of this subscheme is the generically embedded rational curve $\varepsilon(\P^1_k)\subset X$.  Applying the tight approximation property of $\P^1_k$ (for which see \cite[Theorem 10]{Royden} or \cite[Theorem 4.2]{bwtight}) to the normalization of this rational curve and to the strict transform of $\hat{u}$ therefore shows that $\hat{u}$ is close to an algebraic section in the sense of tight approximation. This completes the proof.
\end{proof}

We finally deduce Corollary~\ref{corhyptight} from Theorem~\ref{th2}.

\begin{proof}[Proof of Corollary~\ref{corhyptight}]
When $d=1$ (resp.\ $d=2$), the result follows from \cite[Theorem 4.2]{bwtight} (resp.\ \cite[Example 4.5]{bwtight}).  In all other cases, one can apply Theorems \ref{th2} and~\ref{th:djs}.
\end{proof}

\bibliographystyle{myamsalpha}
\bibliography{holotight}

\providecommand{\bysame}{\leavevmode\hbox to3em{\hrulefill}\thinspace}
\providecommand{\MR}{\relax\ifhmode\unskip\space\fi MR }
\providecommand{\MRhref}[2]{%
  \href{http://www.ams.org/mathscinet-getitem?mr=#1}{#2}
}
\providecommand{\href}[2]{#2}
\begin{thebibliography}{EGA $\text{IV}_4$}

\bibitem[AF23]{AF}
A.~Alarc\'{o}n and F.~Forstneri\v{c}, \emph{Oka-1 manifolds}, preprint 2023,
  arXiv:2303.15855.

\bibitem[AK03]{AraKo}
C.~Araujo and J.~Koll\'ar, \emph{Rational curves on varieties}, Higher
  dimensional varieties and rational points ({B}udapest, 2001), Bolyai Soc.
  Math. Stud., vol.~12, Springer, Berlin, 2003, pp.~13--68.

\bibitem[BK03]{BKC}
J.~Bochnak and W.~Kucharz, \emph{Approximation of holomorphic maps by algebraic
  morphisms}, Proceedings of {C}onference on {C}omplex {A}nalysis
  ({B}ielsko-{B}ia\l a, 2001), vol.~80, 2003, pp.~85--92.

\bibitem[BM96]{BM}
K.~Behrend and Y.~Manin, \emph{Stacks of stable maps and {G}romov-{W}itten
  invariants}, Duke Math. J. \textbf{85} (1996), no.~1, 1--60.

\bibitem[BS49]{BS}
H.~Behnke and K.~Stein, \emph{Entwicklung analytischer {F}unktionen auf
  {R}iemannschen {F}l\"{a}chen}, Math. Ann. \textbf{120} (1949), 430--461.

\bibitem[BW21]{bwtight}
O.~Benoist and O.~Wittenberg, \emph{The tight approximation property}, J. reine
  angew. Math. \textbf{776} (2021), 151--200.

\bibitem[Deb01]{Debarre}
O.~Debarre, \emph{Higher-dimensional algebraic geometry}, Universitext,
  Springer-Verlag, New York, 2001.

\bibitem[dJHS11]{dJHS}
A.~J. de~Jong, X.~He and J.~M Starr, \emph{Families of rationally simply
  connected varieties over surfaces and torsors for semisimple groups}, Publ.
  Math. IH\'ES (2011), no.~114, 1--85.

\bibitem[dJS06]{djs}
A.~J. de~Jong and J.~M. Starr, \emph{Low degree complete intersections are
  rationally simply connected},
  \href{http://www.math.stonybrook.edu/~jstarr/papers/nk1006g.pdf}{\tt
  {http://www.math.stonybrook.edu/\~{}jstarr/papers/nk1006g.pdf}}, 2006.

\bibitem[DLS94]{DLS}
J.-P. Demailly, L.~Lempert and B.~Shiffman, \emph{Algebraic approximations of
  holomorphic maps from {S}tein domains to projective manifolds}, Duke Math. J.
  \textbf{76} (1994), no.~2, 333--363.

\bibitem[EGA $\text{III}_1$]{EGA31}
A.~Grothendieck, \emph{{\'E}l{\'e}ments de g{\'e}om{\'e}trie alg{\'e}brique:
  {III}. \'{E}tude cohomologique des faisceaux coh\'erents, {I}}, Publ. Math.
  IH\'ES (1961), no.~11.

\bibitem[EGA $\text{IV}_3$]{EGA43}
\bysame, \emph{{\'E}l{\'e}ments de g{\'e}om{\'e}trie alg{\'e}brique: {IV}.
  \'{E}tude locale des sch{\'e}mas et des morphismes de sch{\'e}mas, {III}},
  Publ. Math. IH\'ES (1966), no.~28.

\bibitem[EGA $\text{IV}_4$]{EGA44}
\bysame, \emph{\'{E}l\'ements de g\'eom\'etrie alg\'ebrique: {IV}. \'{E}tude
  locale des sch\'emas et des morphismes de sch\'emas, {IV}}, Publ. Math.
  IH\'ES (1967), no.~32.

\bibitem[FL25]{FL}
F.~Forstneri\v{c} and F.~L\'arusson, \emph{Oka-1 manifolds: new examples and
  properties}, Math. Z. \textbf{309} (2025), no.~2, Paper No.~26, 16~p.

\bibitem[For81]{forster}
O.~Forster, \emph{Lectures on {Riemann} surfaces}, Grad. Texts Math., vol.~81,
  Springer, Cham, 1981.

\bibitem[FP97]{FuPa}
W.~Fulton and R.~Pandharipande, \emph{Notes on stable maps and quantum
  cohomology}, Algebraic geometry---{S}anta {C}ruz 1995, Proc. Sympos. Pure
  Math., vol. 62, Part 2, Amer. Math. Soc., Providence, RI, 1997, pp.~45--96.

\bibitem[FP01]{FP}
F.~Forstneri\v{c} and J.~Prezelj, \emph{Extending holomorphic sections from
  complex subvarieties}, Math. Z. \textbf{236} (2001), no.~1, 43--68.

\bibitem[GHS03]{ghs}
T.~Graber, J.~Harris and J.~Starr, \emph{Families of rationally connected
  varieties}, J. Amer. Math. Soc. \textbf{16} (2003), no.~1, 57--67.

\bibitem[GR04]{Stein}
H.~Grauert and R.~Remmert, \emph{Theory of {S}tein spaces}, Classics in
  Mathematics, Springer-Verlag, Berlin, 2004, Translated from the German by
  Alan Huckleberry, Reprint of the 1979 translation.

\bibitem[Has10]{Hassett}
B.~Hassett, \emph{Weak approximation and rationally connected varieties over
  function fields of curves}, Vari\'et\'es rationnellement connexes: aspects
  g\'eom\'etriques et arithm\'etiques, Panor. Synth\`eses, vol.~31, Soc. Math.
  France, Paris, 2010, pp.~115--153.

\bibitem[Hir64]{Hironaka}
H.~Hironaka, \emph{Resolution of singularities of an algebraic variety over a
  field of characteristic zero. {I}}, Ann. of Math. (2) \textbf{79} (1964),
  109--203.

\bibitem[Hir94]{Hirsch}
M.~W. Hirsch, \emph{Differential topology}, Graduate Texts in Mathematics,
  vol.~33, Springer-Verlag, New York, 1994, Corrected reprint of the 1976
  original.

\bibitem[HT06]{HT}
B.~Hassett and Y.~Tschinkel, \emph{Weak approximation over function fields},
  Invent. math. \textbf{163} (2006), no.~1, 171--190.

\bibitem[Kol96]{Kollarbook}
J.~Koll\'ar, \emph{Rational curves on algebraic varieties}, Ergeb. Math.
  Grenzgeb. (3), vol.~32, Springer-Verlag, Berlin, 1996.

\bibitem[Kol04]{Kollarspecialization}
\bysame, \emph{Specialization of zero cycles}, Publ. Res. Inst. Math. Sci.
  \textbf{40} (2004), no.~3, 689--708.

\bibitem[Kuc95]{kucharzRunge}
W.~Kucharz, \emph{The {Runge} approximation problem for holomorphic maps into
  {Grassmannians}}, Math. Z. \textbf{218} (1995), no.~3, 343--348.

\bibitem[KZ24]{kalimanzaidenberg}
S.~Kaliman and M.~Zaidenberg, \emph{Algebraic {G}romov's ellipticity of cubic
  hypersurfaces}, preprint 2024, arXiv:2402.04462.

\bibitem[Min20]{minoccheri}
C.~Minoccheri, \emph{On the arithmetic of weighted complete intersections of
  low degree}, Math. Ann. \textbf{377} (2020), no.~1-2, 483--509.

\bibitem[Roy67]{Royden}
H.~L. Royden, \emph{Function theory on compact {R}iemann surfaces}, J. Analyse
  Math. \textbf{18} (1967), 295--327.

\bibitem[Ser06]{sernesi}
E.~Sernesi, \emph{Deformations of algebraic schemes}, Grundlehren math. Wiss.,
  vol.~334, Berlin: Springer, 2006.

\bibitem[Sta06]{Starr}
J.~M. Starr, \emph{Hypersurfaces of low degree are rationally
  simply-connected}, preprint 2006, arXiv:0602641.

\bibitem[Sta09]{starrarithmeticoverfunctionfields}
\bysame, \emph{Arithmetic over function fields}, Arithmetic geometry. Clay
  Mathematics Institute Summer School Arithmetic Geometry, G\"ottingen,
  Germany, July 17--August 11, 2006, Providence, RI: American Mathematical
  Society (AMS); Cambridge, MA: Clay Mathematics Institute, 2009, pp.~375--418.

\bibitem[{Sta}25]{stacks-project}
The {Stacks project authors}, \emph{The stacks project},
  \url{https://stacks.math.columbia.edu}, 2025.

\bibitem[Ste51]{Steenrod}
N.~Steenrod, \emph{The {T}opology of {F}ibre {B}undles}, Princeton Mathematical
  Series, vol.~14, Princeton University Press, Princeton, NJ, 1951.

\bibitem[Tem12]{TemkinFunctnonemb}
M.~Temkin, \emph{Functorial desingularization of quasi-excellent schemes in
  characteristic zero: the nonembedded case}, Duke Math. J. \textbf{161}
  (2012), no.~11, 2207--2254.

\end{thebibliography}
\end{document}